\documentclass{article}
\usepackage[utf8]{inputenc}
\usepackage[T1]{fontenc}
\usepackage{amssymb}
\usepackage{amsmath}
\usepackage{amsthm}
\usepackage{amsfonts}
\usepackage{esint}
\usepackage{dsfont}
\usepackage{wasysym}
\usepackage{bbm}
\usepackage[all]{xy}
\usepackage{hyperref}
\usepackage{caption}
\usepackage{enumitem}
\usepackage{graphicx}
\usepackage{tikz-cd}
\usepackage{multirow}
\usepackage{array}
\usepackage{geometry}
\newcolumntype{P}[1]{>{\centering\arraybackslash}p{#1}}\newcolumntype{M}[1]{>{\centering\arraybackslash}m{#1}}

\usepackage[style=alphabetic,backend=biber,maxnames=4,minalphanames=4,maxalphanames=4]{biblatex}
\usetikzlibrary{decorations.markings,decorations.pathreplacing,positioning,cd}

\newtheorem{theorem}{Theorem}[section]

\newtheorem{lemma}[theorem]{Lemma}
\newtheorem{proposition}[theorem]{Proposition}
\newtheorem{claim}[theorem]{Claim}

\theoremstyle{definition}
\newtheorem{definition}[theorem]{Definition}
\newtheorem{assumption}[theorem]{Assumption}

\theoremstyle{remark}
\newtheorem*{remark}{Remark}

\newcommand{\R}{\mathbb{R}}
\newcommand{\Z}{\mathbb{Z}}

\newcommand{\Prob}{\mathbb{P}}
\newcommand{\Proba}[1]{ \Prob\left[ #1 \right]}

\newcommand{\norm}[2]{\left\Vert #1 \right\Vert_{#2}}

\newcommand{\scal}[2]{\left<#1,#2\right>}

\newcommand{\ZZ}{\mathbb{Z}^2}

\newcommand{\excur}{\mathcal{E}_\ell}
\newcommand{\connects}{\overset{\excur(f)}{\longleftrightarrow}}
\newcommand{\chem}{\text{chem}}
\newcommand{\diam}{\text{diam}}
\newcommand{\dchem}{d_{\chem}^{\excur(f)}}

\title{Chemical distance in the supercritical phase \\of planar Gaussian fields}
\author{David Vernotte}
\date{December 2023}

\begin{document}

\maketitle

\begin{abstract}
    Our study concerns the large scale geometry of the excursion set of planar random fields: $\excur= \{x\in \mathbb{R}^2 | f(x) \geq -\ell\}$, where $\ell\in \mathbb{R}$ is a real parameter and $f$ is a continuous, stationary, centered, planar Gaussian field satisfying some regularity assumptions (in particular, this study applies to the planar Bargmann-Fock field). It is already known (see for instance \cite{Threshold}) that under those hypotheses there is a phase transition at $\ell_c=0$. When $\ell>0$, we are in a supercritical regime and almost surely $\excur$ has a unique unbounded connected component. We prove that in this supercritical regime, whenever two points are in the same connected components of $\excur$ then,  with high probability, the chemical distance (the length of the shortest path in $\excur$ between these points) is close to the Euclidean distance between those two points
\end{abstract}

\tableofcontents

\section{Introduction}
In this article, we discuss the geometry of the excursion set of a random Gaussian planar field $f$. To motivate this work we begin by a brief reminder of the Bernoulli percolation model on $\ZZ$.

Consider the graph $\ZZ$ (two sites have an edge between them if and only if they are at distance $1$). For each edge $e$ in this graph, we say that $e$ is open with probability $p$ and we keep it in the graph, otherwise we say that $e$ is closed and we remove it from the graph (the choices for the different edges being independent). The new graph we obtain is a random graph that presents several clusters (several connected components). It is known that when $p$ varies from $0$ to $1$ there is a phase transition at some critical parameter $p_c\in ]0,1[$. When $p<p_c$ then almost surely there will only be bounded clusters, but when $p>p_c$ then almost surely there exists exactly one unbounded cluster (and also many bounded ones). A famous theorem of Kesten addresses the value of $p_c$,
\begin{theorem}[Phase transition for Bernoulli percolation \cite{Kesten}]
    \label{thm:kesten}
    For the Bernoulli percolation model on $\ZZ$ we have $p_c=\frac{1}{2},$ that is:
    \begin{itemize}
        \item If $p<1/2$, then almost surely there is no infinite open cluster.
        \item If $p>1/2$, then almost surely there exists a unique infinite cluster.
    \end{itemize}
\end{theorem}
When $p>p_c$, we are in the so-called \textit{supercritical regime} and there exists exactly one unbounded cluster. However this is a "topological" information and it does not quite describe the geometry of this unbounded cluster. An idea to understand the geometry of this unbounded component is to consider the chemical distance.
\begin{definition}
\label{def:dchem_discrete}
In the context of Bernoulli percolation, for $x,y\in \ZZ$ we define the chemical distance $d_\chem$ between $x$ and $y$ as
\begin{equation*}
    d_\chem(x,y) := \left\{\begin{array}{cc}
        \text{min}(\text{length}(\gamma)\ | \ \gamma \text{ is a path of open edges from $x$ to $y$}) &\text{if }x \longleftrightarrow y  \\
         \infty &\text{otherwise} 
    \end{array}\right.,
\end{equation*}
where $x\longleftrightarrow y$ means that $x$ and $y$ are belongs to the same cluster of our random graph, and the length of a path $\gamma$ joining $x$ and $y$ is simply the number of edges in this path.
\end{definition}
\begin{remark}
Note that the chemical distance is \textit{random} in the sense that the chemical distance between two points $x$ and $y$ will vary according to the realization of our graph. When $x$ and $y$ are connected, then $d_\chem(x,y)$ simply represents the graph distance between the two points.
\end{remark}
A natural question is whether this chemical distance has a behaviour close to the behaviour of the Euclidean distance or far from it. This question was addressed by Peter Antal and Agoston Pisztora.
\begin{theorem}[Chemical distance for Bernoulli percolation in supercritical regime \cite{Antal}]
\label{thm:antal}
    If $p>1/2$ then there exists real constants $c,C,\rho>0$ (depending only on $p$) such that
    \begin{equation}
        \label{eq:antal}
        \forall x \in \mathbb{Z}^2, \ \Proba{0 \longleftrightarrow x, d_\chem(0,x)>\rho \norm{x}{}} < C\exp(-c\norm{x}{}).
    \end{equation}
\end{theorem}
\begin{remark}
Note that if $p>p_c$ we have by the FKG inequality and using the unicity of the infinite cluster:
\begin{equation*}
    \Proba{0\longleftrightarrow x} \geq \Proba{0 \longleftrightarrow \infty}\Proba{x \longleftrightarrow \infty} \geq \theta(p)^2>0,
\end{equation*}
where $\theta(p)>0$ is the probability that the cluster of $0$ is unbounded (by definition, the event that $0$ belongs to the infinite cluster is $\{0\longleftrightarrow \infty\}$). So the fact that we have an exponential decay in \eqref{eq:antal} is really due to the constraint on the chemical distance.
\end{remark}
Theorem \ref{thm:antal} answers this question saying that, in the supercritical regime, the chemical distance behaves like the Euclidean distance up to some multiplicative constant with high probability. In some sense,     the infinite cluster will not present very big holes and will not contort itself too much.

Now let us go back to the main matter of this article. Although the above discussion was about a discrete percolation model, we will properly give sense to analogous definition and statements for some continuous percolation models. Those continuous models have received a lot of interest and development recently. Our main result is a similar (although weaker) statement of Theorem \ref{thm:antal} for planar Gaussian random field on $\R^2$.
\subsection{Notations and main result}
Throughout this paper, we consider $f$ a stationary, centered, continuous Gaussian field on $\R^2$ with covariance kernel $\kappa : \R^2 \to \R$ defined as
\begin{equation*}
    \kappa(x) = \mathbb{E}[f(x)f(0)].
\end{equation*}
Since the field $f$ is stationary, the covariance kernel $\kappa$ characterizes the law of the field. In fact, for any $x,y\in \R^2$ we have $ K(x,y) := \mathbb{E}[f(x)f(y)] = \kappa(x-y)$.
One important example of such fields is the so-called Bargmann-Fock field, whose covariance kernel is given by 
\begin{equation}
    \label{eq:cov_bf}
    \kappa_{BF}(x)= \exp\left(-\frac{1}{2}\norm{x}{}^2\right).
\end{equation}
In the expression above and in the following, $\norm{x}{}$ denotes the usual Euclidean norm of $x\in \R^2$.

To state our assumptions on the random field $f$, we will use the notations introduced in \cite{Threshold}, \cite{lcd3}, \cite{Severo}. We begin by introducing the \textit{spectral measure} $\mu$, which is the Fourier transform of $\kappa$. Since $f$ is continuous, such a measure exists by Bochner's theorem (see \cite{NS16} for a more in-depth explanation). We have the following formula for $x\in \R^2$:
\begin{equation*}
    \kappa(x)=\int_{\R^2} e^{i\scal{x}{s}}\mu(ds).
\end{equation*}
In the following, we will always assume that $\mu$ is absolutely continuous with respect to the Lebesgue measure and we write $\rho^2$ as the density of this measure. It is called the \textit{spectral density}. The reason why the existence of the spectral density is a fundamental tool for our and previous analysis (see \cite{Threshold}, \cite{Severo}, \cite{NS16}) is because it is a criterion for obtaining the existence of the white-noise representation of $f$, meaning that we can write $f$ as
\begin{equation}
    \label{eq:white_noise_repr}
    f=q\star W,
\end{equation}
for some $q \in L^2 (\R^2)$ satisfying $q(-x) = q(x)$, where $\star$ denotes the convolution and $W$ is a standard planar white-noise.
The function $q$ can be related to the spectral density $\rho^2$ by the following equation
\begin{equation*}
    q = \mathcal{F}(\rho),
\end{equation*}
where $\mathcal{F}(\cdot)$ denotes the Fourier transform. It is also known that in this case we have
\begin{equation*}
    \kappa = q\star q,
\end{equation*}
where $\star$ again denotes the  convolution. The reverse construction can be done as follows, if $q\in L^2(\R^2)$ is given such that $q(-x)=q(x)$, then $f:= q\star W$ will be a stationary centered planar Gaussian field with covariance kernel $\kappa = q\star q$.  For the Bargmann-Fock field with covariance kernel given by \eqref{eq:cov_bf} the function $q_{BF}$ can be computed exactly:
\begin{equation}
\label{eq:qBF}
    \begin{array}{ccccc}
         q_{BF} & : & \R^2 & \to & \R  \\
         & & x & \mapsto & \sqrt{\frac{2}{\pi}}e^{-\norm{x}{}^2}.
    \end{array}
\end{equation}
Several assumptions very similar to those in \cite{Threshold} or \cite{Severo} will be made in the present article.
\begin{assumption}
\label{assumption:a1} The function $q$ has the following properties:
\begin{itemize}
    \item $q\in L^2(\R^2).$
    
    \item $q(x)$ is isotropic (it only depends on $\norm{x}{}$).
    \item $q \geq 0.$
    \item $q$ is not identically equal to the zero function.
    \item The support of the spectral measure $\mu$ contains a non empty open set.
\end{itemize}
\end{assumption}
\begin{assumption}[Regularity, depends on some $m\geq 3$]\label{assumption:a2}
The function $q$ has the following regularity:
\begin{itemize}
\item $q\in \mathcal{C}^m(\R^2).$
\item $\forall |\alpha|\leq m,\  \partial^\alpha q \in L^2(\R^2).$
\end{itemize}
In the second condition, $\alpha=(\alpha_1,\alpha_2)$ is an element of $\mathbb{N}^2$ and by definition $|\alpha|=\alpha_1+\alpha_2.$
\end{assumption}
\begin{assumption}[Strong positivity]
\label{assumption:a3}
The function $q$ is non negative: $\forall x\in \R^2,\ q(x)\geq 0.$
\end{assumption}
\begin{assumption}[Decay of correlation, depends on some $\beta>0$]
\label{assumption:a4}
    There exists $C>0$ such that for any $x\in \R^2\setminus\{0\}$,
    \begin{equation*}
        \max\left\{q(x), \norm{\nabla q(x)}{}\right\} \leq C\norm{x}{}^{-\beta}.
    \end{equation*}
\end{assumption} We make a few comments about these assumptions.
Basically, the first element of Assumption \ref{assumption:a1} and the fact that $q$ is invariant under sign change (which is implied by the isotropy of $q$) allow us to rigorously define $q\star W$. The regularity of $q$ in Assumption \ref{assumption:a2} entails by dominated convergence that $\kappa$ will be $\mathcal{C}^{2m}$ differentiable, hence it allows us to see $f$ as a $\mathcal{C}^{m-1}$ differentiable function (in fact a $\mathcal{C}^{m-1}$ modification of $q\star W$). The Assumption \ref{assumption:a3} is called the \textit{strong positivity} hypothesis. The main purpose of this assumption is to make use of the FKG inequality. Note that it was conjectured in \cite{Threshold} that Assumption \ref{assumption:a3} could be replaced by the weaker hypothesis $q\star q \geq 0$ for their purpose. We also remark that all these hypothesis are verified (for any $\beta>0$ and $m\geq 3$) by the Bargmann-Fock field whose function $q$ is given by \eqref{eq:qBF}.

Now that we have introduced our planar random field $f$, we present the percolation model associated to it. For $\ell\in \mathbb{R}$ we define the excursion set at level $\ell$ as the random set
\begin{equation*}
    \excur(f) := \{x\in \R^2 | f(x)\geq -\ell\}.
\end{equation*}
This excursion set has been studied with regard to percolation theory \cite{BG16}, \cite{HA_critical} and is to be seen as the analog to the random graph in Bernoulli percolation. In particular, it was proved in \cite{HA_critical} for the Bargmann-Fock field and later in \cite{Threshold} and \cite{NOFKG} for more general Gaussian fields that there is a sharp transition at $\ell_c=0$, this is an analog of Kesten's Theorem \ref{thm:kesten}.
\begin{theorem}[Phase transition for the Bargmann-Fock field \cite{HA_critical}, \cite{Threshold}]
    \label{thm:kesten_continuous}
    Assume that $q$ satisfies Assumptions \ref{assumption:a1}, \ref{assumption:a2} for some 
    $m\geq 3$, \ref{assumption:a3} and \ref{assumption:a4} for some $\beta>2$.
    If $\ell<0$, then almost surely $\excur(f)$ only have bounded connected components.
    If $\ell>0$, then almost surely $\excur(f)$ admits a unique unbounded connected component.
\end{theorem}
In particular, if $\ell>0$ we are in the supercritical regime and we observe the existence of a unique unbounded component in the excursion set. The same question of understanding the geometry of this set arises. Let us introduce a few notations and definitions.
\begin{definition}
For any $\ell\in \mathbb{R}$ and $x,y\in \R^2$ we introduce the event $\left\{x\connects y\right\}$ that the two points $x$ and $y$ are in the same connected component of $\excur(f)$.
\end{definition}
\begin{definition}
\label{def:dchem_continous}
For any $\ell\in \mathbb{R}$ and $x,y\in \R^2$ we introduce the chemical distance between $x$ and $y$ as
\begin{equation*}
    \dchem(x,y) := \left\{\begin{array}{cc}
        \inf\{\text{length}(\gamma)\ | \ \gamma\in \Gamma(x,y, \excur(f))\} & \text{if }x\connects y \\
        \infty & \text{otherwise} 
    \end{array}\right.,
\end{equation*}
where $\Gamma(x,y,\excur(f))$ denotes the set of continous and rectifiable paths $\gamma : [0,1] \to \R^2$ such that $\gamma(0)=x,\  \gamma(1)=y$ and such that $\gamma$ takes values in $\excur(f)$, and $\text{length}(\gamma)$ denotes the Euclidean length of the path $\gamma$.
\end{definition}
\begin{remark}
The fact that we can find some rectifiable path $\gamma$ between two points $x$ and $y$ when they are connected stems from the fact that the field $f$ is smooth and the excursion set $\excur(f)$ will be a 2-dimensional smooth submanifold with boundary of $\R^2$ (see Lemma \ref{lemma:smooth} for a more rigorous statement).
\end{remark}
\begin{remark}
Definition \ref{def:dchem_continous} is to be understood as the analog for continuous percolation of Definition \ref{def:dchem_discrete}.
\end{remark}
Finally we can state our main result.
\begin{theorem}
    \label{thm:principal}
    Assume that $q$ satisfies Assumptions \ref{assumption:a1}, \ref{assumption:a2} for some $m\geq 3$, \ref{assumption:a3} and \ref{assumption:a4} for some $\beta>2$. Then, for any $\ell>0$, for any $0<\delta<1$ we have a constant $C>0$ (depending on $\ell, q, \beta, m, \delta$) such that for all $x\in \R^2$ with $\norm{x}{}>3$ the following is satisfied:
    \begin{equation}
        \label{eq:main_thm1}
        \Proba{0\connects x \text{ and }\dchem(0,x)>\norm{x}{}\log^{\frac{3}{2}+\delta}\norm{x}{}}\leq C\frac{\log^{\frac{m-1}{2}}\norm{x}{}}{\norm{x}{}^m}.
        \end{equation}
\end{theorem}
That is, we show that in some sense the chemical distance between two points behaves almost like the Euclidean distance when those two points are far away.
\begin{remark}
For the Bargmann-Fock field, Assumption \ref{assumption:a2} is verified for any $m\geq 3$. Hence Theorem \ref{thm:principal} implies that the probability in \eqref{eq:main_thm1} has super-polynomial decay. However, remark that the constant $C>0$ in the statement of the theorem depends a priori on the value of $m$.
\end{remark}

\subsection{A few words about Theorem \ref{thm:principal} and strategy of the proof}
Here are some comments about our main theorem. To begin with, the resemblance between the two Theorems \ref{thm:antal} and \ref{thm:principal} is clear, the statements of both theorems can be reformulated in words as follows: in the supercritical phase, the chemical distance behaves almost like the Euclidean distance with high probability. Yet, one may notice two main differences in the statements of the two theorems. One is that we observe an additional logarithm term in our theorem that wasn not there in the theorem of \cite{Antal} for Bernoulli percolation. This is due to essentially two factors. First is the fact that while there is a minimal scale in Bernoulli percolation (an edge is of length $1$, there is no smaller scale), this is a priori not the case in continuous percolation, the field $f$ could locally oscillate and contort a lot, creating an unexpected large chemical distance in a box of size $1$. Second is that the field $f$ can have long-range correlation, $f(x)$ and $f(y)$ are correlated even when $x$ and $y$ are far away.
Another difference is that while we obtain an exponential decay for the probability in Theorem \ref{thm:antal}, we only get a bound that may decrease very slowly for small values of $m$. This again is due to the same problem of possible oscillations on arbitrary small scales which prevents us to have an easy control on the chemical distance around a point.

Now, let us present the strategy of the proof. It differs quite a lot from the one for Bernoulli percolation in \cite{Antal}. By isotropy of the field we may assume that we want to connect the origin $(0,0)$ to some point $(x,0)$. First of all we will build with high probability what we call a \textit{global structure} around those two points and that is contained in a "thin" rectangle (see Figure \ref{fig:global_structure}). This structure is composed of the following parts: two circuits contained in $\excur(f)$, one around $(0,0)$ and one around $(x,0)$ and a connection (also in $\excur(f)$) between those two circuits. All this structure is contained in a thin rectangle of length linear in $x$. The purpose of this global structure is that under the event $\left\{(0,0)\connects (0,x)\right\}$ we will be able to build a path in $\excur(f)$ joining $(0,0)$ and $(x,0)$ that is contained within the thin rectangle. Such a path is a good candidate for having a quasi linear length in term of Euclidean distance since it is contained in this rectangle. The details of the construction of this structure are given in Section \ref{sec:globa_struct}. In particular we build the global structure at a harder level $\ell'<\ell$, this will allow us to thicken a little the circuits and the connection between the circuits to ensure that two points within the global structure have a reasonable chemical distance between them. Next, we consider the portions of the path between $(0,0)$ and $(x,0)$ before it reaches the global structure. Because of the existence of arbitrary small scales where the field $f$ can oscillate, it would be possible to observe such portions of path of arbitrary long length. We did not find an easy argument to handle this, so we propose an argument relying on a Kac-Rice formula and a deterministic geometric argument to counter this problem in Section \ref{sec:local_struct}, the technical aspects of this section will be delayed to the end of Section \ref{sec:proof_of_main_thm}. The estimate we obtain is, however, not good enough to guarantee an exponential decay in the probability, this is where we lose the strong decay. With these two ideas we will be ready to prove Theorem \ref{thm:principal} in Section \ref{sec:proof_of_main_thm}.

\begin{remark}
We believe that the main theorem may hold in higher dimension but we would need additional arguments like for example the property of strong percolation.
\end{remark}

\paragraph{Acknowledgements: }I am very grateful to my PhD advisor Damien Gayet for introducing me to this subject as well as for his many advices. I would also like to thank Hugo Vanneuville for helpful conversations.

\section{Construction of the global structure}
\label{sec:globa_struct}
In this section, we build a global structure around the two points $(0,0)$ and $(x,0)$ that will allow us to shorten existing paths in the excursion set joining those two points. We basically build this global structure using well-known crossing events and we ensure those crossing are thick enough by working with a discretization of our field.
\subsection{Discretization of the field}
We begin with a classical definition of a discretization of the field $f$ that will be used later on. The motivavtion is the following: in our context of continuous percolation, even if one knows that there exists a path in $\excur(f)$ between two points $x$ and $y$ and also knows that this path is contained in a box of some fixed size, this does not  give any deterministic control on the Euclidean length of this path. In fact, contrary to Bernouilli percolation where the minimal scale is $1$ (the length of an edge is $1$), a new problem that arises in continuous percolation is that the field could oscillates a lot, the nodal set $\mathcal{Z}_\ell(f) := \{f=-\ell\}$ could contort itself a lot, causing an unexpected large chemical distance between two points. In this section we are interested in building a large structure around two points $(0,0)$ and $(x,0)$. This will be done by building said structure at a harder level $0<\ell'<\ell$ and then using a general argument to say that if the structure exists at level $\ell'$ then at the easier level $\ell$ it will still exist and will be thickened a little. That will allow us to have control on the chemical distance between any two points of this global structure.
Although we could work only with the field $f$ and study constraints on the second derivative of $f$ for such a good behaviour to happen. We prefer to make use of a discretization $f^\varepsilon$ of the field $f$, and apply concentration estimates to recover information about $f$ from $f^\varepsilon$. The precise discretization procedure is given in the following definition.
\begin{definition}
\label{def:fepsilon}
Suppose a function $g : \R^2 \to \R$ is given, as well as a parameter $\varepsilon>0$. We define $g^\varepsilon : \R^2 \to \R$ to be the function defined on $\R^2$ by
\begin{equation*}
    \forall x\in \varepsilon \Z^2, \forall y\in \left[-\frac{\varepsilon}{2}, \frac{\varepsilon}{2}\right[^2,\ g^\varepsilon(x+y):=g(x).
\end{equation*}
\end{definition}\noindent
The function $g^\varepsilon$ is piecewise constant on each small square of side-width $\varepsilon$ and centered at some $x\in \varepsilon\ZZ$. If $g$ is continuous, the function $g^\varepsilon$ is intuitively a good approximation of the function $g$ when $\varepsilon$ is small.
This intuition can me made rigorous in the following sense for our random field $f$.

\begin{proposition}[\cite{Severo} Proposition 2.1]
\label{prop:concentration_inequality}
Assume that $q$ satisfies Assumptions \ref{assumption:a1}, \ref{assumption:a2} for some $m \geq 3$ and \ref{assumption:a4} for some $\beta > 1$. Then there exists some constant $c>0$ (depending only on $q$) such that for all $\varepsilon>0$, for all $s\geq \varepsilon$
    \begin{equation}
     \Proba{\sup_{y\in \R^2, \ \norm{y}{}\leq 1}|f^\varepsilon(y)-f(y)| \geq s} \leq \exp(-cs^2\varepsilon^{-2}). \label{eq:concentration_3}
    \end{equation}
\end{proposition}
This proposition ensures that we can recover information about $f$ from information about $f^\varepsilon$. However, as we will see in the upcoming sections, this need to be done by using a little sprinkling of the level $\ell$.
\subsection{Crossing events}
We now briefly remind the reader of the definition of a crossing event. In the following, we consider $\mathcal{R}=[0,a]\times [0,b]$ a rectangle in the ambient space $\R^2$. When considering the intersection $\excur(f)\cap \mathcal{R}$ we obtain a set with a certain number of connected components. We say that there is a \textit{crossing} of $\mathcal{R}$ by $\excur(f)$ if among those connected components there is at least one that intersects both $\{0\}\times [0,b]$ and $\{a\}\times [0,b]$. We define the event $\{f\in \text{Cross}_\ell(\mathcal{R})\}$ as the event that the rectangle $\mathcal{R}$ is crossed by $\excur(f)$.
\begin{figure}
\centering
\includegraphics[width=10cm]{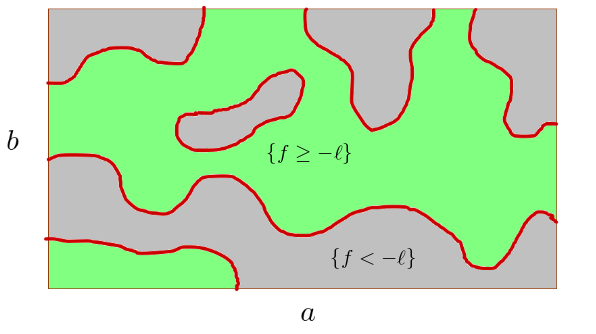}
\caption{Representation of the event $\text{Cross}_\ell(\mathcal{R})$}.
    \label{fig:crossing_event}
\end{figure}In \cite{Threshold}, estimates are obtained for the probability of such events in the supercritical regime.

\begin{theorem}[\cite{Threshold}]
	\label{thm:prob_crossing_event}
    Assume $q$ satisfies Assumptions \ref{assumption:a1}, \ref{assumption:a2} for some $m\geq 3$, \ref{assumption:a3} and \ref{assumption:a4} for some $\beta > 2$. If $\ell>0$ and $a,b>0$ are given, there exists a constant $c>0$ such that
    \begin{equation*}
        \forall \lambda\geq 1,\ \Proba{f\in \emph{\text{Cross}}_\ell(\lambda a,\lambda b)}\geq 1-e^{-c\lambda}.
    \end{equation*}
\end{theorem}

Although not stated in \cite{Threshold}, it can be easily shown that the same statement holds for $f^\varepsilon$ (see Definition \ref{def:fepsilon}) if $\varepsilon$ is small enough.
\begin{proposition}
\label{prop:prob_crossin_event_epsilon}
Assume $q$ satisfies Assumptions \ref{assumption:a1}, \ref{assumption:a2} for some $m \geq 3$, \ref{assumption:a3} and \ref{assumption:a4} for some $\beta > 2$.
    If $\ell>0$ and $a,b>0$ are given, there exist constants $c>0, \varepsilon_0>0$ such that,
    \begin{equation*}
        \forall 0<\varepsilon<\varepsilon_0,\ \forall \lambda\geq 1,\ \Proba{f^\varepsilon \in \emph{\text{Cross}}_\ell(\lambda a,\lambda b)}\geq 1-e^{-c\lambda}.
    \end{equation*}
\end{proposition}
The proof of this proposition is similar to the one in \cite{Threshold}. We need to use the concentration inequality of Proposition \ref{prop:concentration_inequality} and the fact that the event $\{f^\varepsilon \in \text{Cross}_\ell(\mathcal{R})\}$ has probability close to the probability of $\{f\in \text{Cross}_\ell(\mathcal{R})\}$ if $\mathcal{R}$ is fixed and $\varepsilon$ is close to $0$. However, for the sake of clarity in the paper we do not write the proof here. A corollary of this theorem that will be useful later on is stated bellow.
\begin{proposition}
    \label{prop:long_crossing_event_prob}
    Assume $q$ satisfies Assumptions \ref{assumption:a1}, \ref{assumption:a2} for some $m \geq 3$, \ref{assumption:a3} and \ref{assumption:a4} for some $\beta > 2$. Suppose $\ell>0$ is given as well that some functions $L : \mathbb{R}_+\to \mathbb{R}_+$ and $l : \mathbb{R}_+\to \mathbb{R}_+$ that satisfy:
    \begin{itemize}
        \item $\exists C>0,\  \forall x \geq C,\ 1\leq l(x)\leq L(x),$
        \item $l(x)\xrightarrow[x\to\infty]{}\infty$.
    \end{itemize}
    Then, there exist constants $c>0, \varepsilon_0 >0$ such that,
    \begin{equation*}
        \forall 0<\varepsilon < \varepsilon_0,\  \forall x\geq C,\  \Proba{f^\varepsilon \in \emph{\text{Cross}}_\ell(L(x),l(x))}\geq 1 - \frac{4L(x)}{l(x)}e^{-cl(x)}.
    \end{equation*}
\end{proposition}
\begin{proof}
    The idea of the proof is basically to assemble a certain number of rectangles of dimension $2l(x)\times l(x)$ to observe the apparition of a crossing of dimension $L(x)\times l(x)$. Since $q$ is isotropic, the probability of crossing an horizontal rectangle is the same as crossing the rotated rectangle vertically. Fix $x\geq C$. For $i\in \{0,1,2,\dots, \lceil L(x)/l(x)\rceil-1\}$ we denote the rectangles,
    \begin{equation*}
            H_i := [l(x)i, l(x)(i+1)]\times [0,l(x)].
    \end{equation*}
    And for $j\in \{1,2,\dots, \lceil L(x)/l(x)\rceil-2\}$ we consider the rectangles,
    \begin{equation*}
            V_j := [l(x)j,l(x)(j+1)]\times [0, 2l(x)].
    \end{equation*}
    Note that if $\lceil L(x)/l(x)\rceil\leq 2$ then we do not consider any vertical rectangle $V_j$. Consider $\mathcal{H}_i$ the event that $H_i$ is crossed horizontally by $\mathcal{E}_\ell(f^\varepsilon)$, and $\mathcal{V}_i$ the event that $V_i$ is crossed vertically by $\mathcal{E}_\ell(f^\varepsilon)$ (see Figure \ref{fig:fig2_2D}).
    \begin{figure}
        \centering
        \includegraphics[width=14cm]{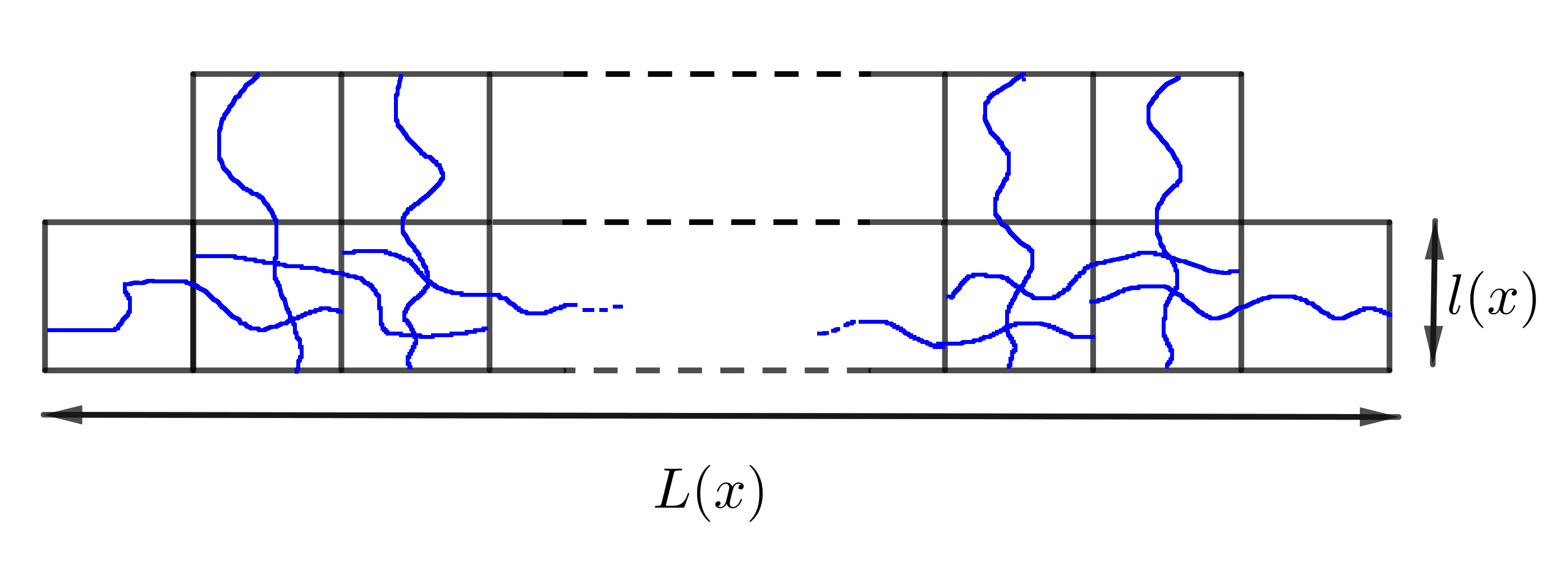}
        \caption{Illustration of the proof of Proposition \ref{prop:long_crossing_event_prob}.}
        \label{fig:fig2_2D}
    \end{figure}
    If we choose $a=2$ and $b=1$ in Proposition \ref{prop:prob_crossin_event_epsilon}, we get $\varepsilon_0>0,R_0,c>0$ such that if $\lambda \geq 1$ and $0<\varepsilon<\varepsilon_0$ then
    \begin{equation*}
            \Proba{f^\varepsilon\in \text{Cross}_\ell(2\lambda, \lambda)}\geq 1 - e^{-c\lambda}.
    \end{equation*}
    By isotropy, stationarity and the fact that $l(x)\geq 1$ we have for all $i,j$
    \begin{equation*}
            \Proba{\mathcal{H}_i}= \Proba{\mathcal{V}_j}=\Proba{f^\varepsilon \in \text{Cross}_\ell(2l(x),l(x))}\geq 1 - e^{-cl(x)}.
    \end{equation*}
    Now observe that the occurrence of all the events $\mathcal{H}_i$ and $\mathcal{V}_j$ implies that the rectangle $[0,L(x)]\times [0,l(x)]$ is crossed horizontally in $\mathcal{E}_\ell(f^\varepsilon)$ (see Figure \ref{fig:fig2_2D}). It only remains to apply a union bound, and noticing that $2\lceil L(x)/l(x)\rceil \leq 4L(x)/l(x)$ to conclude the proof.
\end{proof}
\subsection{The global structure}
In this subsection, we properly define the \textit{global structure} and we show that it will exist with high probability. In this section we suppose that $x>3$ is fixed as well as a parameter $\delta>0$ (that is independent of everything else including $x$ and that can thought of as small). We assume that $\varepsilon\in ]0,1[$ is fixed (it will be later be fixed as a function of $x$). Also we consider that $\ell>0$ is given, and we define the so-called \textit{harder level}
\begin{equation*}
    \ell'=\ell/2.
\end{equation*}
As explained in the introduction, we will work at the harder level $\ell'$ with $f^\varepsilon$ to recover information of $f$ at level $\ell$.
Now we make the following definition.
\begin{definition}
\label{def:thin_rectangle}
For $x>3$ and $\delta>0$, we define the quantities
\begin{itemize}
\item $l(x) = \log^{1+\delta}(x),$
\item $L(x) = x+l(x).$
\end{itemize}
We also define the \textit{thin rectangle} $H$,
\begin{equation}
    \label{eq:thin_rectangle}
    H = \left[-\frac{l(x)}{2}, -\frac{l(x)}{2}+L(x)\right]\times\left[-\frac{l(x)}{2}, \frac{l(x)}{2}\right].
\end{equation}
\end{definition}
Now the global structure is represented in Figure \ref{fig:global_structure}. It is composed of a crossing of the thin rectangle and of two circuits in the thin rectangle, one around $(0,0)$ and one around $(x,0)$. The crossing of the thin rectangle must also be a connection between the two circuits.
\begin{figure}
\centering
\includegraphics[width=15cm]{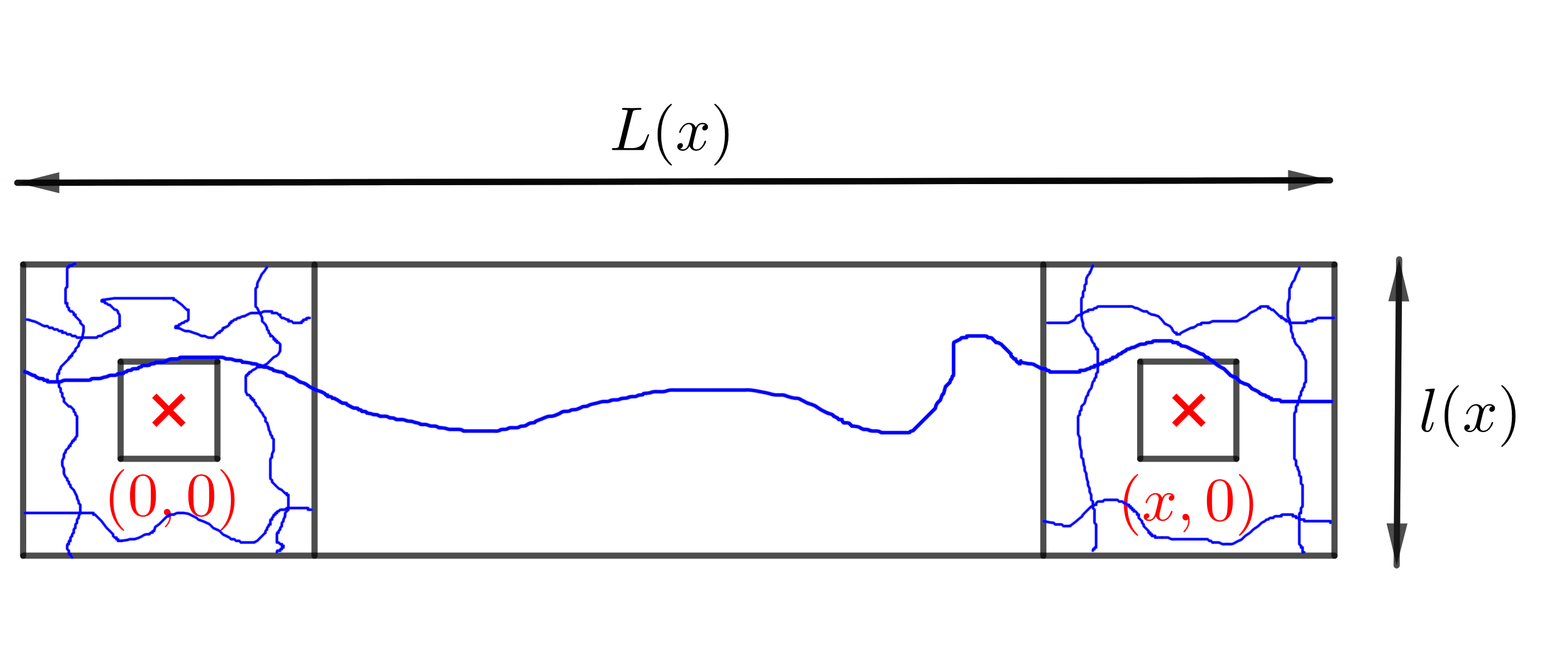}
\caption{Representation of the global structure.}
\label{fig:global_structure}
\end{figure}
To be rigorous we properly define an event that implies the existence of this global structure. Consider the following rectangles
\begin{itemize}
\item $H_1 = [-l(x)/2, -l(x)/6]\times [-l(x)/2,l(x)/2],$
\item $H_2 = [l(x)/6, l(x)/2]\times [-l(x)/2,l(x)/2],$
\item $H_3 = [-l(x)/2, l(x)/2]\times [-l(x)/2,-l(x)/6],$
\item $H_4 = [-l(x)/2, l(x)/2]\times [l(x)/6,l(x)/2].$
\end{itemize}
And for $i\in \{1,2,3,4\}$ we also consider their translated by $(x,0)$, $H'_i = (x,0)+H_i$. We denote $\mathcal{H}_i$ (resp $\mathcal{H}'_i$) the event that $H_i$ (resp $H'_i$) is crossed by $\mathcal{E}_{\ell'}(f^\varepsilon)$ lengthwise. By stationarity and isotropy we have for all $i\in \{1,2,3,4\}$
\begin{equation*}
    \Proba{\mathcal{H}_i} =\Proba{\mathcal{H}'_i} = \Proba{f^\varepsilon \in \text{Cross}_{\ell'}(l(x),l(x)/3)}.
\end{equation*}
We also consider the big thin rectangle $H = [-l(x)/2,x+l(x)/2]\times [-l(x)/2,l(x)/2]$, and we denote $\mathcal{H}$ the event that this rectangle is crossed by $\mathcal{E}_{\ell'}(f^\varepsilon)$ lengthwise. Finally, the event we consider is
\begin{equation}
    \label{eq:event_gx}
  \mathcal{G}^{(1)}_{x,\varepsilon} := \mathcal{H}\cap \bigcap_{i\in \{1,2,3,4\}}(\mathcal{H}_i \cap \mathcal{H}'_i).  
\end{equation}

This event $\mathcal{G}^{(1)}_{x,\varepsilon}$ is important for our purpose in the sense that it occurrence implies that if there is a path from $(0,0)$ to $(x,0)$ in $\mathcal{E}_{\ell}(f)$ then we are close to be able to find a similar path that stays in the thin rectangle $H$. This is made precise in the following lemma.
\begin{lemma}
	\label{lemma:global_struct1}
    With the above notations, on the event $\mathcal{G}^{(1)}_{x,\varepsilon}$, if there exists a path joining $(0,0)$ and $(x,0)$ in $\mathcal{E}_\ell(f)$, then we can find a path $\gamma$ joining $(0,0)$ and $(x,0)$ that is contained in $H$, and that can be written as $\gamma=\gamma_1\cup\gamma_2\cup\gamma_3$ with the further conditions:
    \begin{itemize}
    \item $\gamma_1$ is a path contained in $\excur(f)\cap [-l(x)/2,l(x)/2]\times [-l(x)/2,l(x)/2].$
    \item $\gamma_3$ is a path contained in $\excur(f)\cap [x-l(x)/2,x+l(x)/2]\times [-l(x)/2,l(x)/2].$
    \item $\gamma_2$ is a path contained in $\mathcal{E}_{\ell'}(f^\varepsilon)\cap H$.
\end{itemize}
\end{lemma}
\begin{proof}
The proof is completely trivial once it is understood that a path from $(0,0)$ to $(x,0)$ must intersect both circuits of the global structure. We do not further detail the proof.
\end{proof}
We also argue that the event $\mathcal{G}^{(1)}_{x,\varepsilon}$ has high probability.
\begin{lemma}
\label{lemma:global_structure_high_prob}
Assume $q$ satisfies Assumptions \ref{assumption:a1}, \ref{assumption:a2} for some $m\geq 3$, \ref{assumption:a3} and \ref{assumption:a4} for some $\beta>0$. Suppose $\ell>0$ is given and $\ell' =\frac{\ell}{2}$, suppose also that $\delta>0$ is given. Then there exist constant $\varepsilon_0, c, C>0$ such that
\begin{equation*}
    \forall 0<\varepsilon < \varepsilon_0, \forall x\geq 3, \Proba{\mathcal{G}^{(1)}_{x,\varepsilon}}\geq 1-Ce^{-cl(x)},
\end{equation*}
where we recall that $l(x)=\log(x)^{1+\delta}$ was defined in Definition \ref{def:thin_rectangle}.
\end{lemma}
\begin{proof}
Recall the definition of $\mathcal{G}^{(1)}_{x,\varepsilon}$ in \eqref{eq:event_gx}.
We apply a union bound to estimate the probability of the complementary of this event:
\begin{align*}
    \Proba{(\mathcal{G}^{(1)}_{x,\varepsilon})^c}\leq \Proba{\mathcal{H}^c} + \sum_{i=1}^4 \Proba{(\mathcal{H}_i)^c}+\Proba{(\mathcal{H}'_i)^c}.
\end{align*}
Now, we can apply Proposition \ref{prop:long_crossing_event_prob} with $\ell'>0$, $l(x)=\log(x)^{1+\delta}$ and $L(x)=x+l(x)$, to get $\varepsilon'_0 > 0$ and $c'>0$ such that for all $0<\varepsilon <\varepsilon'_0$ and all $x\geq 3$ we have
\begin{equation*}
    \Proba{\mathcal{H}^c} \leq \frac{4L(x)}{l(x)}e^{-c'l(x)}.
\end{equation*}
We can also apply Proposition \ref{prop:prob_crossin_event_epsilon} with $\ell'>0$, $a=1$ and $b=1/3$, to get $\varepsilon''_0>0$ and $c''>0$ such that for all $0<\varepsilon <\varepsilon''_0$ and all $x\geq 3$ we have
\begin{equation*}
    \Proba{(\mathcal{H}_1)^c} \leq e^{-c''l(x)}.
\end{equation*}
Note that we used the fact that $l(x)\geq 1$ when $x\geq 3$. Moreover the previous property concerning $\mathcal{H}_1$ extend to all $\mathcal{H}_i$ and $\mathcal{H}'_i$ by stationarity and isotropy of the field. We now conclude by choosing $\varepsilon_0 = \min(\varepsilon'_0,\varepsilon''_0)$ and adjusting the constants $c,C$ (note that we use the fact that $\frac{4L(x)}{l(x)}e^{-c'l(x)} = O(e^{-cl(x)})$ for any $0<c<c'$).
\end{proof}

Now we present how we recover information about $f$ from information about $f^\varepsilon$. Consider the event
\begin{equation}
    \label{eq:event_gprimex}
    \mathcal{G}^{(2)}_{x,\varepsilon} := \{\norm{f-f^\varepsilon}{\infty, H}<\frac{\ell}{2}\}.
\end{equation}
 This event ensures that in the thin rectangle $H$, then $f$ is very close to $f^\varepsilon$ (this will have high probability thanks to Proposition \ref{prop:concentration_inequality}). Under this event we have the following:
\begin{lemma}
\label{lemma:global_struct2}
Assume that the event $\mathcal{G}^{(2)}_{x,\varepsilon}$ holds. There exists a universal constant $C>0$ such that, for any points $y_1,y_2\in H$, if these two points are connected within $H$ by a path in $\mathcal{E}_{\ell'}(f^\varepsilon)$, then they are also connected by a path in $\mathcal{E}_\ell(f)$ of length at most $\frac{C}{\varepsilon}\text{Area}(H)$.
\end{lemma}
\begin{proof}
Take any two points $y_1,y_2\in H$ connected within $H$ by a path in $\mathcal{E}_{\ell'}(f^\varepsilon)$. Because $f^\varepsilon$ is piecewise constant we can consider all squares of the form $z+ \left[-\frac{\varepsilon}{2}, \frac{\varepsilon}{2}\right[^2$ that intersect $H$ (where $z\in \varepsilon\Z^2$). Since $x\geq 3$, we see that $l(x)$ and $L(x)$ are bounded away from $0$ and there exists a universal constant $C>0$ such that there are at most $C\frac{Area(H)}{\varepsilon^2}$ such small squares intersecting $H$. Now we say that such a small square $z+\left[-\frac{\varepsilon}{2}, \frac{\varepsilon}{2}\right[^2$ is open if $f^\varepsilon(z)\geq -\ell'= -\frac{\ell}{2}$ and closed otherwise. Because there is a path in $\mathcal{E}_{\ell'}(f^\varepsilon)$ between $y_1$ and $y_2$, we can find a path of adjacent open squares joining the square containing $y_1$ and the square containing $y_2$ (two squares $z+\left[-\frac{\varepsilon}{2}, \frac{\varepsilon}{2}\right[^2$ $z'+\left[-\frac{\varepsilon}{2}, \frac{\varepsilon}{2}\right[^2$ are said to be adjacent if $\norm{z-z'}{\infty}=\varepsilon$, so that a square has at most eight adjacent squares).
On the event $\mathcal{G}^{(2)}_{x,\varepsilon}$, we see that on every open square we have
\begin{equation*}
    f\geq f^\varepsilon-\norm{f-f^\varepsilon}{\infty, H}\geq -\frac{\ell}{2}-\frac{\ell}{2}\geq -\ell.
\end{equation*}
Thus, we can now find a path in $\mathcal{E}_\ell(f)$ joining $y_1$ and $y_2$. This path will simply goes through some open squares in straight line. Moreover we can build such a path so that each small square will be crossed at most once. Hence, we found a path in $\mathcal{E}_\ell(f)$ joining $y_1$ and $y_2$ of length at most $\sqrt{2}\varepsilon \times C\frac{\text{Area}(H)}{\varepsilon^2}= \sqrt{2}C \frac{\text{Area}(H)}{\varepsilon}$.  
\end{proof}
We conclude this section, showing that the event $\mathcal{G}^{(2)}_{x,\varepsilon}$ has high probability.
\begin{lemma}
\label{lemma:event_gtilde_high_prob}
Assume that $q$ satisfies Assumptions \ref{assumption:a1}, \ref{assumption:a2} for some $m\geq 3$, \ref{assumption:a4} for some $\beta>1$. Assume that $\ell>0$ is fixed, then there exist constants $C,c>0$ such that for all $0<\varepsilon<\ell'=\frac{\ell}{2}$ and for all $x\geq 3$, we have
\begin{equation*}
    \Proba{\mathcal{G}^{(2)}_{x,\varepsilon}} \geq 1- Cx\log(x)^{1+\delta}\exp(-c\varepsilon^{-2}).
\end{equation*}
\end{lemma}
\begin{proof}First,if $x\geq 3$ we can cover the thin rectangle $H$ with at most $C\text{Area}(H)$ balls of radius $1$, where $C>0$ is a universal constant. Adjusting the constant $C$ if needed, we see that we can cover $H$ with $Cx\log(x)^{1+\delta}$ balls of radius $1$. By stationarity of the field an applying a union bound we see that
\begin{equation*}
\Proba{(\mathcal{G}^{(2)}_{x,\varepsilon})^c} \leq Cx\log(x)^{1+\delta}\Proba{\sup_{\norm{y}{}\leq 1}|f^\varepsilon(y)-f(y)| \geq \frac{\ell}{2}}.
\end{equation*}
It only remains to apply Proposition \ref{prop:concentration_inequality} with $s=\frac{\ell}{2}$ to get the conclusion.
\end{proof}

\section{Local control of the chemical distance around a point}
\label{sec:local_struct}
In this section we present an argument to control the behavior of the chemical distance locally around a point. This argument relies on a Kac-Rice formula and on a deterministic geometric argument.
\subsection{A deterministic argument}
In the following, we take $\mathcal{E}\subset \R^2$ a $\mathcal{C}^1$ differentiable 2-dimensional submanifold of $\R^2$ with boundary. We also take a parameter $R\geq 1$ (it will soon depend on $x$), we consider a box $B$ of the form $B = [-R,R]^2$ and we further assume that $\mathcal{E}$ intersects $B$ but that $\partial \mathcal{E}$ can only intersect $\partial B$ transversely. Under those hypotheses we consider $\mathcal{C}$ one of the connected component of $B\cap \mathcal{E}$. Since $\mathcal{E}$ is smooth and the boundary of $B$ is piecewise $\mathcal{C}^1$, the following definition makes sense.
\begin{figure}
    \centering
    \includegraphics[width=10cm]{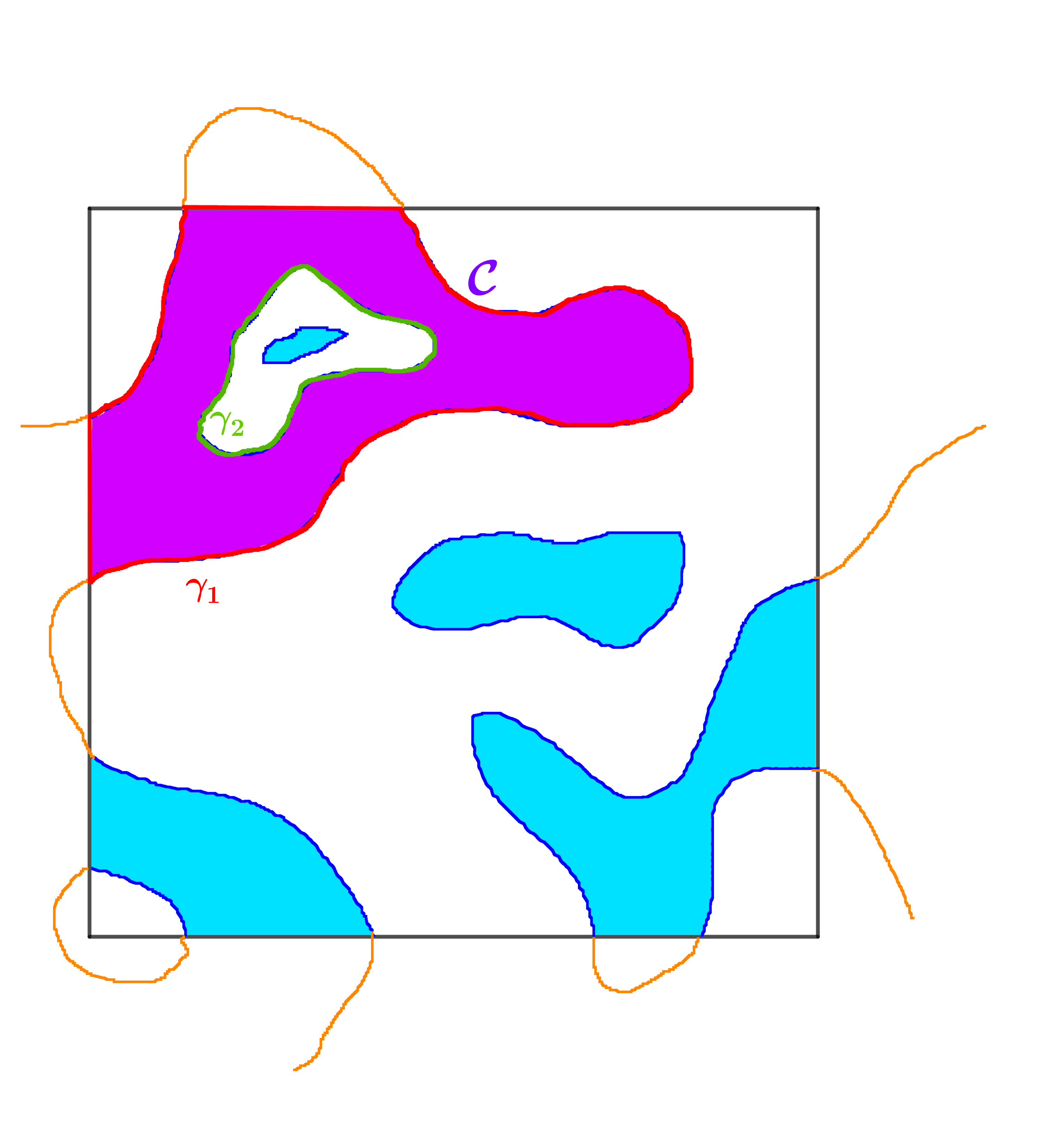}
    \caption{Illustration of $\mathcal{C}$ (in purple) and the boundary $\partial \mathcal{C} = \gamma_1 \sqcup \gamma_2$. The other clusters are in light blue.}
    \label{fig:my_label}
\end{figure}
\begin{definition}
If $x,y\in \mathcal{C}$ we define the chemical distance between $x$ and $y$ in $\mathcal{C}$ as,
\begin{equation*}
d_{\text{chem}}^{\mathcal{C}}(x,y) := \inf_{\gamma\in \Gamma(x,y,\mathcal{C})}\text{length}(\gamma),
\end{equation*}
where $\Gamma(x,y,\mathcal{C})$ is the set of all continuous rectifiable paths $\gamma$ from $x$ to $y$ contained in $\mathcal{C}$, and by $\text{length}(\gamma)$ we mean the Euclidean length of the path $\gamma$.
\end{definition}
We remark that if $\mathcal{C}$ is taken to be a connected component of $\mathcal{E}_\ell(f)$, then this chemical distance is the usual chemical distance as defined in the introduction (see Definition \ref{def:dchem_continous}).
\begin{definition}
\label{def:diam_chem}
We also define the chemical diameter of $\mathcal{C}$ as
\begin{equation*}
\text{diam}_{\text{chem}}(\mathcal{C}) := \sup_{x,y\in \mathcal{C}}d^\mathcal{C}_{\chem}(x,y).
\end{equation*}
\end{definition}
That is we look at two points in $\mathcal{C}$ that are as far away as possible for the distance $d_\chem^{\mathcal{C}}$.
\begin{remark}
 Note that since $\mathcal{E}$ is a smooth manifold with boundary and since $B$ is closed, then $\mathcal{C}$ is a closed and bounded set of $\R^2$ hence a compact set. In particular, one can easily show that the supremum in Definition \ref{def:diam_chem} is in fact a maximum (in general this maximum will not be reached by a unique pair of points $x,y$).
\end{remark}
We will be interested in the topological boundary of $\mathcal{C}$. It is known that the topological boundary of $\mathcal{E}$ is a union of topological circles (Jordan curves) and topological lines. When intersected with the ball $B = [-R,R]^2$, we see that the boundary of $\mathcal{C}$ is a union of disjoints Jordan curves that are piecewise $\mathcal{C}^1$ (there is no degeneracy since we assume that $\partial \mathcal{E}$ can only intersect $\partial B$ transversely). Note that some of those Jordan curve can possibly contains parts of the boundary of $B$, and that even if $\mathcal{E}$ is connected, this is not necessarily the case of $\mathcal{E}\cap B$ (see Figure \ref{fig:my_label}). We now write the boundary of $\mathcal{C}$ as
\begin{equation*}
    \partial\mathcal{C} = \gamma_1 \sqcup \gamma_2 \sqcup \dots \sqcup \gamma_n,
\end{equation*}
where the $\gamma_i$ are disjoint Jordan curves. This allows the following definition.
\begin{definition}
With the notations above, we define the length of $\partial \mathcal{C}$ as
\begin{equation*}
    \text{length}(\partial \mathcal{C}) := \sum_{i=1}^n \text{length}(\gamma_i),
\end{equation*}
where $\text{length}(\gamma_i)$ denotes the Euclidean length of the Jordan curve $\gamma_i$ (which is well defined since $\gamma_i$ is piecewise $\mathcal{C}^1$).
\end{definition}
Now our key lemma is the following.
\begin{lemma}
\label{lemma:grid}
With the notations above, we have
\begin{equation*}
\text{\emph{diam}}_\text{\emph{chem}}(\mathcal{C}) \leq 2 \text{\emph{length}}(\partial \mathcal{C}).
\end{equation*}
\end{lemma}
\begin{remark}
We do not think that the factor $2$ in the statement of the lemma is optimal, however it will be enough for our purpose.
\end{remark}
The proof of Lemma \ref{lemma:grid} is long and technical and can be skipped for a first reading. We delay the proof of this lemma to the end of Section \ref{sec:proof_of_main_thm}.
\begin{figure}
\centering
\includegraphics[width=12cm]{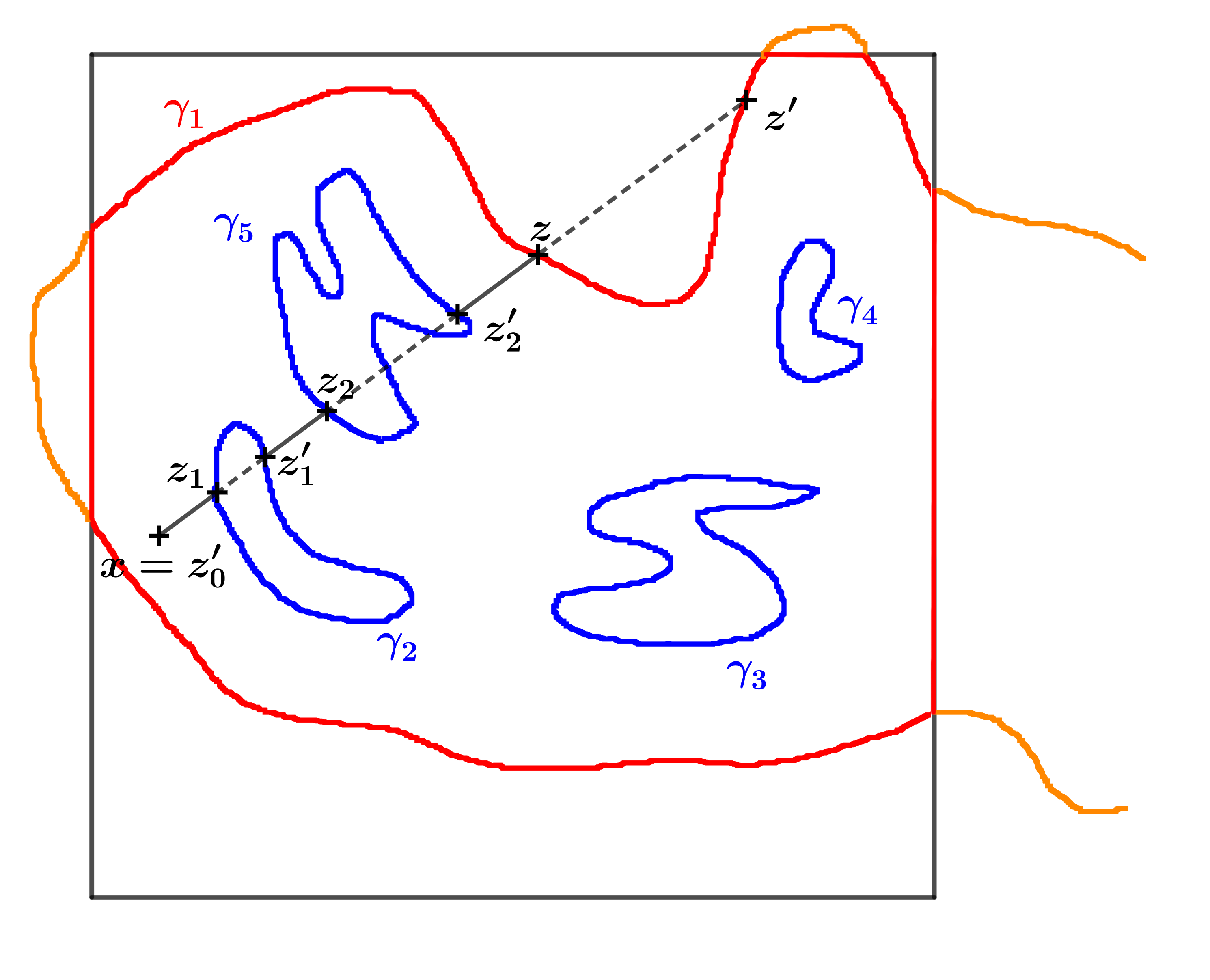}
\caption{Illustration of Lemma \ref{lemma:grid}.}
\label{fig:lemma_geometric}
\end{figure}

\subsection{Control of the chemical distance around a point}
Now we go back to our random field $f$. With the notation of the previous section, we want to take $\mathcal{E} := \mathcal{E}_\ell(f)$ (so that the set $\mathcal{E}$ is random). However to make this rigorous we need to state the following well-known lemma:
\begin{lemma}[Lemma A.9 in \cite{quasi_independance}]
\label{lemma:smooth}
Assume that $q$ satisfies Assumptions \ref{assumption:a1}, \ref{assumption:a2} for some $m\geq 3$ and \ref{assumption:a4} for some $\beta > 2$, then if $\ell\in \mathbb{R}$  and $R>0$ are fixed the following is true almost surely:
\begin{enumerate}
    \item The sets $\mathcal{E}_\ell(f) = \{x\in \R^2\ |\ f(x)\geq -\ell\}$ and $\mathcal{E}_{-\ell}(-f)=\{x \in \R^2\ |\ f(x)\leq -\ell\}$ are two 2-dimensional smooth submanifold of $\R^2$ with boundary. Moreover their boundary is the same and is precisely the set $\mathcal{Z}_\ell(f) := \{x\in \R^2\ |\ f(x)=-\ell\}.$
    \item The set $\mathcal{Z}_\ell(f)$ can only intersect $\partial B_R$ transversally.
\end{enumerate}
\end{lemma}
Using this lemma, we see that it is legitimate to work with $\mathcal{E}=\excur(f)$. Now take $R\geq 1$ and $B= [-R,R]^2$.
As discussed in the previous section, we can say that $\excur(f)\cap B$ will be a union of several connected components, $\mathcal{C}_1, \dots, \mathcal{C}_n$.
\begin{definition}
\label{def:sB}
With the above notation we define $S(B)$ to be the random variable
\begin{equation*}
    S(B):= \sum_{i=1}^n \diam_\chem(\mathcal{C}_i).
\end{equation*}
\end{definition}
The reason we introduce $S(B)$ is because having control over $S(B)$ guarantees a control on the chemical distance between any two points connected in $\mathcal{E}_\ell(f)\cap B$. In fact, if two points $x,y$ are connected within $\mathcal{E}_\ell(f)\cap B$ then they must belong to the same connected component $\mathcal{C}_i$ for some $i$ and the chemical distance $\dchem(x,y)$ is smaller or equal than $\diam_\chem(\mathcal{C}_i)$ the chemical diameter of $\mathcal{C}_i$ which itself is smaller or equal than $S(B)$. Now we explain how we will achieve control over $S(B)$.
\begin{proposition}
\label{prop:kac_rice}
Assume that $q$ satisfies Assumptions \ref{assumption:a1}, \ref{assumption:a2} for some $m\geq 3$, \ref{assumption:a3} and \ref{assumption:a4} for some $\beta>2$. Then for all $0\leq k\leq m-1 $, we have a constant $C_k>0$ (depending on $k$, $\ell$ and $q$) such that for all $R\geq 1$ and $B=[-R,R]^2$,
\begin{equation*}
\mathbb{E}[S(B)^k] \leq C_kR^{2k}.
\end{equation*}
\end{proposition}
\begin{proof}
The proof is a combination of our Lemma \ref{lemma:grid} and an application of a Kac-Rice formula. Consider the bounded open set $U=\mathring{B} = ]-R,R[^2$. Since our field $f$ is almost surely $\mathcal{C}^1$ and non degenerate we can consider the length of $\{x\in \R^2 \ |\ f(x)=-\ell\}\cap U$. We denote $L(f, \ell, U) := \text{length}(\{x\in \R^2\ |\ f(x)=-\ell\}\cap U)$ this length. In \cite{azais} this quantity is studied 
and it is shown how to obtain a close formula for the expectation of $L(f,\ell,U)$ in term of the covariance matrix of $f$. Recent results in \cite{gass2023number} and \cite{letendre2023multijet} show that the moment of order $k$ of $L(f,\ell,B)$ is finite.
More precisely, under our assumptions $f$ is a $\mathcal{C}^{m-1}$ differentiable function, hence using Theorem 1.6 in \cite{letendre2023multijet} or Theorem 1.5 in \cite{gass2023number} we obtain for any $0\leq k \leq m-1$
\begin{equation}
    \label{eq:kr1}
    C_k := \mathbb{E}[L(f,\ell,[0,1]^2)^k]  < \infty.
\end{equation}
For simplicity, we assume that $R\geq 1$ is an integer (otherwise replace $R$ by $\lceil R\rceil$ in the following). We cover the box $B=[-R,R]^2$ by $N:= 4R^2$ boxes of the form $x+[0,1]^2$ where $x\in \ZZ$ (note that some of these boxes will overlap on one side). We denote $(B_j)_{1\leq j\leq N}$ the collection of these boxes (the order in which we take the boxes is not relevant). Also, notice that we deterministically have:
\begin{equation}
    \label{eq:kr2}
    L(f,\ell,B) \leq \sum_{j=1}^N L(f,\ell, B_j).
\end{equation}
In the following we denote
\begin{equation}
    \label{eq:ka}
    X_j := L(f,\ell,B_j).
\end{equation}
By stationarity of the field we see that all $X_j$ have the same law, furthermore using \eqref{eq:kr1} is rephrased as
\begin{equation}
    \label{eq:kr3}
    \forall 1\leq j\leq N,\ \mathbb{E}[(X_j)^k] = C_k <\infty.
\end{equation}
Taking expectation in \eqref{eq:kr2} and using in succession the multinomial formula and the Hölder inequality we get
\begin{align}
    \label{eq:kac_rice2}
    \mathbb{E}[L(f,\ell,B)^k] & \leq \mathbb{E}\left[\left(\sum_{j=1}^N X_j\right)^k\right] \nonumber\\
    &= \sum_{k_1+\dots + k_N=k}\frac{k!}{k_1!\dots k_N!}\mathbb{E}[(X_1)^{k_1}\dots (X_N)^{k_N}] \nonumber\\
    &\leq \sum_{k_1+\dots + k_N=k}\frac{k!}{k_1!\dots k_N!}\mathbb{E}[(X_1)^{k}]^{\frac{k_1}{k}}\dots \mathbb{E}[(X_N)^{k}]^{\frac{k_N}{k}} \nonumber\\
    & = \sum_{k_1+\dots + k_N=k}\frac{k!}{k_1!\dots k_N!}C_k  = C_k N^k = 4^kC_k R^{2k}.
\end{align}
Now we apply Lemma \ref{lemma:grid} to see that  $\text{diam}_\chem(\mathcal{C}_i)\leq 2\text{length}(\partial C_i)$. By definition of $S(B)$ it only remains to control the total sum $\sum_{i=1}^n \text{length}(\partial \mathcal{C}_i)$. This sum is less or equal than the length of $\partial (\mathcal{E}_\ell(f)\cap B)$ which itself is less or equal than the length of $\partial B$ plus the length of $\partial (\mathcal{E}_\ell(f) \cap B)$.
Hence have the following inequalities
\begin{align}
    \label{eq:kac_rice3}
S(B) & = \sum_{i=1}^n \diam_\chem(\mathcal{C}_i)  \leq 2\sum_{i=1}^n\text{length}(\partial \mathcal{C}_i)\nonumber\\
& \leq 2 \left(\text{length}(\partial B)+L(f,\ell,B)\right) = 2\left(8R + L(f,\ell,B)\right).
\end{align}
Taking expectation in this inequality and using \eqref{eq:kac_rice2} we get
\begin{align*}
    \mathbb{E}[S(B)^k]& \leq \mathbb{E}\left[\left(2(8R+L(f,\ell,B))\right)^k\right] \\
    & \leq 2^k\sum_{i=0}^k \binom{k}{i}4^iC_iR^{2i}(8R)^{k-i} \\
    & \leq C'_k R^{2k},
\end{align*}
for some constant $C'_k>0$ that depend on $(C_i)_{0\leq i\leq k}$.  This yields the conclusion of our proposition since this constant only depends on $k$, $q$ and $\ell$.
\end{proof}

\section{Proof of the main theorem}
\label{sec:proof_of_main_thm}
In this section, we prove the main theorem.
\begin{proof}[Proof of Theorem \ref{thm:principal}]
Recall that by isotropy of the field, it is enough to look at the connection event $\{(0,0)\connects(x,0)\}$.
In the proof we take $x>3$ and some $\delta>0$ and we consider the following events:
\begin{itemize}
\item $\mathcal{C}_x:= \left\{(0,0)\connects (x,0)\right\}.$
\item $\mathcal{D}_x := \left\{\dchem((0,0),(x,0)) > C_1x\log^{3/2+2\delta}(x)\right\},$ where $C_1>0$ is a constant to be fixed later.
\end{itemize}
We also take $\varepsilon\in ]0,1[$ (for the moment $\varepsilon$ is arbitrary). Recall Definition \ref{def:thin_rectangle} of the thin rectangle $H$ of dimension $L(x)\times l(x)$ where 
\begin{equation*}
    l(x)=\log^{1+\delta}(x) \text{ and } L(x) = x+l(x),
\end{equation*}
and the events $\mathcal{G}^{(1)}_{x,\varepsilon}$ and $\mathcal{G}^{(2)}_{x,\varepsilon}$ introduced by \eqref{eq:event_gx} and \eqref{eq:event_gprimex} in Section \ref{sec:globa_struct} as well as the definition of the discretization $f^\varepsilon$ (see Definition \ref{def:fepsilon}). Recall also the definition of the harder level
\begin{equation*}
    \ell'=\frac{\ell}{2}.
\end{equation*}

The event $\mathcal{G}^{(1)}_{x,\varepsilon}$ implies the existence of the so called global structure. Under the event $\mathcal{G}^{(1)}_{x,\varepsilon}\cap \mathcal{C}_x$, there exists a path in $\excur(f)$ from $(0,0)$ to $(x,0)$ and we can apply Lemma \ref{lemma:global_struct1} to this path. The lemma provides us a path $\gamma$ from $(0,0)$ to $(x,0)$ that is contained in $H$. Furthermore, the lemma states that $\gamma$ can be written as $\gamma= \gamma_1\cup\gamma_2\cup \gamma_3$ where, $\gamma_1$ (resp $\gamma_3$) is a path contained in $\excur(f)$ and in a box of size $l(x)$ around $(0,0)$ (resp $(x,0)$), and $\gamma_2$ is contained in $\mathcal{E}_{\ell'}(f^\varepsilon)$. Now by Lemma \ref{lemma:global_struct2}, we see that on the event $\mathcal{G}^{(2)}_{x,\varepsilon}$, the path $\gamma_2$ can be replaced by a path $\gamma'_2$ in $\mathcal{E}_\ell(f)$ of length at most $C_2x\log^{1+\delta}(x)\varepsilon^{-1}$ where $C_2>0$ is a universal constant.

Moreover, with the notations of Definition \ref{def:sB}, we introduce the random variables
\begin{align*}
S_{l(x)} & = S([-l(x)/2,l(x)/2]\times [-l(x)/2,l(x)/2]),\\
S'_{l(x)} &= S([x-l(x)/2,x+l(x)/2]\times [-l(x)/2,l(x)/2]),
\end{align*}
and we define the events,
\begin{align*}
\mathcal{L}_x &:= \{S_{l(x)} \leq x\log^{3/2+2\delta}(x)\}, \\
\mathcal{L}'_x  &:= \{S'_{l(x)} \leq x\log^{3/2+2\delta}(x)\}.
\end{align*}
 Under the event $\mathcal{C}_x \cap \mathcal{G}^{(1)}_{x,\varepsilon} \cap \mathcal{L}_x \cap \mathcal{L}'_x$, we see that the paths $\gamma_1$ and $\gamma_3$ obtained earlier can be replaced by $\gamma_1'$ and $\gamma_3'$ paths of $\excur(f)$ of length at most $x\log^{3/2+2\delta}(x)$.

Now we choose
\begin{equation*}
    \varepsilon=\log^{-(1/2+\delta)}(x).
\end{equation*}
With the previous discussion we see that under $\mathcal{C}_x \cap \mathcal{G}^{(1)}_{x,\varepsilon}\cap \mathcal{G}^{(2)}_{x,\varepsilon} \cap \mathcal{L}_x\cap \mathcal{L}'_x$ we can find a path $\gamma' = \gamma_1'\cup \gamma_2' \cup \gamma_3'$ in $\excur(f)$ joining $(0,0)$ and $(x,0)$ of length at most:
\begin{align*}
    \text{length}(\gamma')&\leq x\log^{3/2+2\delta}(x)+C_2x\log^{3/2+2\delta}(x)+x\log^{3/2+2\delta}(x)\\ &\leq (2+C_2)x\log^{3/2+2\delta}(x).
\end{align*}
That way, the chemical distance between $(0,0)$ and $(x,0)$ is at most $(2+C_2)x\log^{3/2+2\delta}(x)$. Hence, choosing $C_1 = 2+C_2$ in the definition of the event $\mathcal{D}_x$ we observe that we have the following
\begin{equation}
\label{eq:m1}
\Proba{\mathcal{C}_x\cap \mathcal{D}_x} \leq \Proba{(\mathcal{G}^{(1)}_{x,\varepsilon})^c}+\Proba{(\mathcal{G}^{(2)}_{x,\varepsilon})^c}+\Proba{(\mathcal{L}_x)^c}+\Proba{(\mathcal{L}'_x)^c}.
\end{equation}
It remains to control all four term in the right-hand side of \eqref{eq:m1}. For the first term, we simply apply Lemma \ref{lemma:global_structure_high_prob} so that there exist constants $\varepsilon_0>0$, $c>0$ and $C>0$ (depending only on $q$ and $\ell$) such that if $x\geq 3$ and $\varepsilon = \log^{-(1/2+\delta)}(x)<\varepsilon_0$ we have:
\begin{equation}
\label{eq:m2}
\Proba{(\mathcal{G}^{(1)}_{x,\varepsilon})^c}\leq C\exp(-c\log^{1+\delta}(x)).
\end{equation}
For the second term, we apply Lemma \ref{lemma:event_gtilde_high_prob} so that there exist constants $\tilde{c}, \tilde{C}>0$ such that if $x\geq 3$ and $\varepsilon=\log^{-(1/2+\delta)}(x)<\ell'$ we have:
\begin{equation}
\label{eq:m3}
\Proba{(\mathcal{G}^{(2)}_{x,\varepsilon})^c}\leq \tilde{C}\exp(-\tilde{c}\log^{1+2\delta}(x)).
\end{equation}
Note that we can adjust the constants $C$ and $\tilde{C}$ so that the inequalities \eqref{eq:m2} and \eqref{eq:m3} hold for all $x\geq 3$.
For the third and fourth term, we use the stationarity of the field, Proposition \ref{prop:kac_rice} for some $0\leq k\leq m-1$ as well as Markov inequality to get
\begin{align}
\label{eq:m4}
\Proba{(\mathcal{L}_x)^c} = \Proba{({\mathcal{L}'}_x)^c} & \leq \Proba{S_{l(x)}\geq x\log^{3/2+2\delta}(x)} \nonumber\\
& \leq \frac{\mathbb{E}[(S_{l(x)})^k]}{\left(x\log^{3/2+2\delta}(x)\right)^k} \nonumber\\
& \leq  C_k\frac{\log^{2k(1+\delta)}(x)}{x^k\log^{3k/2+2k\delta}(x)} \nonumber\\&
\leq C_k \frac{\log^{k/2}(x)}{x^k}.
\end{align}
Finally, using \eqref{eq:m2}, \eqref{eq:m3}, \eqref{eq:m4} (for $k=m-1$) in \eqref{eq:m1} we conclude the proof of the theorem since $\delta>0$ is arbitrary.
\end{proof}
It only remains to prove Lemma \ref{lemma:grid}.
\begin{proof}[Proof of Lemma \ref{lemma:grid}]
The proof is long and somewhat technical but the idea behind is actually simple. We suggest the reader to have Figure \ref{fig:lemma_geometric} in mind when reading the proof.

First, recall that in the setting of Lemma \ref{lemma:grid}, $\mathcal{E}\subset \R^2$ denotes a two-dimensional smooth $\mathcal{C}^1$ differentiable manifold with boundary, such that $\partial \mathcal{E}$ only intersects $B=[-R,R]^2$ transversally. We recall that $\mathcal{C}$ is one of the connected components of $B\cap \mathcal{E}$. We also recall that the boundary of $\mathcal{C}$ is a union of disjoint Jordan curves and we write $\partial \mathcal{C}$ as $\partial \mathcal{C}=\gamma_1\sqcup \gamma_2 \sqcup \dots \sqcup \gamma_n$ for some $n\geq 1$, and where every $\gamma_i$ is a Jordan curve that is piecewise $\mathcal{C}^1$.
For each curve $\gamma_i$ we denote $\text{Int}(\gamma_i)$ and $\text{Ext}(\gamma_i)$ the interior and the exterior of $\gamma_i$, so that $\R^2 = \text{Int}(\gamma_i)\sqcup \gamma_i \sqcup \text{Ext}(\gamma_i).$ For simplicity we also make use of the notations $\overline{\text{Int}}(\gamma_i) := \gamma_i \cup \text{Int}(\gamma_i)$ and $\overline{\text{Ext}}(\gamma_i) := \gamma_i \cup \text{Ext}(\gamma_i)$. To have a clear understanding of the structure of $\partial \mathcal{C}$ we state the following claim.
\begin{claim}
\label{claim:1}
We can permutate the curves $\gamma_i$ so that we can assume that for all $2\leq j \leq n$, $\gamma_j \subset \text{Int}(\gamma_1)$. Furthermore we have $\mathcal{C} = \overline{\text{Int}}(\gamma_1)\setminus \left(\cup_{j=2}^n \text{Int}(\gamma_j)\right).$
\end{claim}
\begin{proof}[Proof of Claim \ref{claim:1}]
First we observe that given any of the curve $\gamma_i$ we have are in exactly one of the two following cases:
\begin{itemize}
    \item Either $\mathcal{C}\cap \text{Int}(\gamma_i)=\emptyset$ and we say that $\gamma_i$ is a \textit{hole} in $\mathcal{C}$.
    \item Either $\mathcal{C}\cap \text{Ext}(\gamma_i)=\emptyset$ and we say that $\gamma_i$ is a \textit{contour} of $\mathcal{C}$.
\end{itemize}
In fact, both options can not failed at the same time otherwise we would find $x\in \mathcal{C}\cap \text{Int}(\gamma_i)$ and $y\in \mathcal{C}\cap \text{Ext}(\gamma_i)$. Since $\mathcal{C}$ is connected, we would find a path in $\mathcal{C}$ connecting $x$ and $y$. This path would cross $\gamma_i$ contradicting the fact that $\gamma_i \subset \partial \mathcal{C}$. Also, both options can not be satisfied simultaneously, otherwise we would have $\mathcal{C}\subset \gamma_i$, which is a contradiction since $\mathcal{C}$ is a two dimensional manifold with boundary (while 
$\gamma_i$ is a one dimensional Jordan curve). 

Now, given any pair of two differents curves $\gamma_i$ and $\gamma_j$, since they do not intersect we have three disjoint options:
\begin{enumerate}
    \item $\gamma_i \subset \text{Int}(\gamma_j)$.
    \item $\gamma_j \subset \text{Int}(\gamma_i)$.
    \item $\gamma_i \subset \text{Ext}(\gamma_j) \text{ and }\gamma_j \subset \text{Ext}(\gamma_i)$.
\end{enumerate}
With this observation we argue that there can not be two different contours of $\mathcal{C}$. By contradiction, assume $\gamma_i$ and $\gamma_j$ are two contours of $\mathcal{C}$. We then have $\mathcal{C}\subset \overline{\text{Int}}(\gamma_j)$ and $\mathcal{C}\subset \overline{\text{Int}}(\gamma_j)$. This implies that option 3 for $\gamma_i$ and $\gamma_j$
 is not possible otherwise $\mathcal{C}$ would be the empty set. By symmetry, we can assume that $\gamma_i \subset \text{Int}(\gamma_j)$. But this is also not possible. In fact, if $\gamma_i\subset \text{Int}(\gamma_j)$ then $\gamma_j \subset \text{Ext}(\gamma_i)$. However $\gamma_i$ is a contour and we have $\mathcal{C}\cap \text{Ext}(\gamma_i)=\emptyset$. Also, the inclusion $\gamma_j\subset \mathcal{C}$ is always true. These two inclusions yield $\gamma_j \subset \text{Ext}(\gamma_i)\cap \mathcal{C} = \emptyset$ which is a contradiction. Finally we see that there can at most one curve that is a contour of $\mathcal{C}$.
 
 Similarly we can easily show that if $\gamma_1$ if a contour and $\gamma_i$ a hole then $\gamma_i \subset \text{Int}(\gamma_1)$, and that is $\gamma_i, \gamma_j$ are two holes, then $\gamma_i \subset \text{Ext}(\gamma_j) \text{ and }\gamma_j \subset \text{Ext}(\gamma_i)$.
 
 It remains to show that there exists one contour (say $\gamma_1$) of $\mathcal{C}$ but this follows from the fact that $\mathcal{C}$ is bounded while $\bigcap_{i=1}^n \overline{\text{Ext}}(\gamma_i)$ is unbounded. Finally the formula $\mathcal{C} = \overline{\text{Int}}(\gamma_1)\setminus \left(\cup_{j=2}^n \text{Int}(\gamma_j)\right)$ is just a rewriting of the fact that $\gamma_1$ is the unique contour and that $\gamma_j$ for $j\geq 2$ are holes of $\mathcal{C}$.
\end{proof}
We go back to the proof of Lemma \ref{lemma:grid}. By the conclusion of Claim \ref{claim:1}, without loss of generality we say that $\gamma_1$ is the contour of $\mathcal{C}$ and that $\gamma_2, \dots, \gamma_n$ are the holes of $\mathcal{C}$. We fix $z'\in \gamma_1$ a \textit{landmark} that is fixed until the end of the proof. We claim the following:
\begin{equation}
\label{eq:claim_length}
\forall x\in \mathcal{C},\  d^\mathcal{C}_\chem(x,z')\leq \text{length}(\gamma_1)+\frac{1}{2}\sum_{j=2}^n \text{length}(\gamma_j).
\end{equation}
Using the triangular inequality and taking the supremum over all points $x,y\in \mathcal{C}$, we easily see that this Claim \eqref{eq:claim_length} implies the conclusion of Lemma \ref{lemma:grid}. It only remains to prove the above claim \eqref{eq:claim_length}. Take any $x\in \mathcal{C}$ and consider $S$ the oriented segment in $\mathbb{R}^2$ joining $x$ to our landmark $z'$. More precisely we define
\begin{equation*}
    x(t) := (1-t)x+tz'\text{ for }t\in [0,1],
\end{equation*}
    and
    \begin{equation*}
    S:=[x,z']=\{x(t)\ |\ t\in [0,1]\}.
\end{equation*}
We proceed to an exploration of the segment $S$ starting from $x$ and heading towards $z'$. In this exploration we consider $\tau$ the first time that the segment $S$ intersects the contour $\gamma_1$.
\begin{equation*}
    \tau := \inf\{t\in [0,1]\ |\ x(t)\in \gamma_1\}.
\end{equation*}
Note that since $z'=x(1)$ belongs to $\gamma_1$ then $\tau$ is well defined. We denote $z := x(\tau)$ the first intersection point between $S$ and $\gamma_1$. We have the following claim:
\begin{claim}
\label{claim:2}
The infimum in the definition of $\tau$ is in fact a minimum and we have $z=x(\tau)$ belongs to $\gamma_1$.
\end{claim}
\begin{proof}[Proof of Claim \ref{claim:2}.]
This is a classic argument that relies on the fact that $\gamma_1$ is a closed set.
Consider $(t_n)_{n\geq 0}$ a sequence of elements of $\{t\in [0,1]\ |\ x(t)\in \gamma_1\}$ such that $t_n \xrightarrow[n\to \infty]{}\tau$. Since $t\mapsto x(t)$ is continuous we have
\begin{align*}
    x(t_n) \xrightarrow[n\to \infty]{}x(\tau).
\end{align*}
Furthermore the set $\gamma_1$ is a closed set and $x(t_n)\in \gamma_1$ for all $n\geq 0$, this implies $x(\tau)\in \gamma_1$.
\end{proof}
We return to the proof of Lemma \ref{lemma:grid}. We have divided the segment $S$ into two segments, the segment $[x,z]$ and the segment $[z,z']$ (one of these segments can possibly be reduced to a point if $\tau=0$ or $\tau=1$). Since $z$ and $z'$ both belong to $\gamma_1$, the chemical distance between $z$ and $z'$ is less than half the length of $\gamma_1$. In fact, we can simply consider a path $\gamma$ that follows the boundary $\gamma_1$ from $z$ towards $z'$. Hence we have
\begin{equation}
    \label{eq:18}
 d_\chem^\mathcal{C}(z,z')\leq \frac{1}{2}\text{length}(\gamma_1).
\end{equation}
Now with \eqref{eq:18} in hands, we see with the triangular inequality that to obtain \eqref{eq:claim_length} it only remains to show that the chemical distance between $x$ and $z$ is less than $\frac{1}{2}\sum_{i=1}^n \text{length}(\gamma_i)$.
To prove this, we again proceed to an exploration of the segment $[x,z]$.
We will define several sequences by recurrence,
\begin{itemize}
 \item $(t_k)_{k\geq 1} \in [0,1]^{\mathbb{N}^*},$
 \item $(i_k)_{k\geq 1} \in \{1,\dots,n\}^{\mathbb{N}^*}$\text{ (where $n$ is the number of curves $\gamma_i$)},
 \item $(z_k)_{k\geq 1} \in (\R^2)^{\mathbb{N}^*},$
 \item $(t'_k)_{k\geq 1} \in [0,1]^{\mathbb{N}^*},$
 \item $(z'_k)_{k\geq 0} \in (\R^2)^\mathbb{N}.$
\end{itemize}
First we set $z'_0=x$. Then we define
\begin{align*}
    t_1 := \inf\{t\in [0,\tau]\ |\ x(t) \in \bigcup_{j=1}^n \gamma_j\}.
\end{align*}
We see that $t_1$ is well defined since $z=x(\tau)\in \gamma_1$. Furthermore, by a similar argument to the one in Claim \ref{claim:2},  we see that the infimum in the definition of $t_1$ is in fact a minimum and we denote $z_1 := x(t_1)$, so that $z_1$ belongs to $\bigcup_{j=1}^n \gamma_j$. We also define $i_1$ to be the integer $i_1\in \{1,\dots,n\}$ such that $z_1\in \gamma_{i_1}.$ Then we consider
\begin{align*}
    t'_1 := \sup\{t\in [t_1,\tau]\ |\ x(t) \in \gamma_{i_1}\}.
\end{align*}
And we define $z'_1 = x(t'_1)$ (again we see that this supremum is in fact a maximum). A more informal way to describe this construction is to say that $z_1$ is the first intersection point between $S$ and one of the curves $\gamma_j$, $i_1$ is the number of the curve we intersect for the first time, and $z'_1$ is the last intersection point between $S$ and $\gamma_{i_1}$. Note that it is possible to have $z_1=z'_1$ (if $t_1=t'_1$). If $i_1=1$ then necessarily by the definition of $z$ we see that $z_1=z'_1=z$ and we end the exploration. Otherwise, we continue the exploration. Suppose that $(z_k,z'_k,i_k,t_k,t'_k)$ are well defined for some $k\geq 1$, we consider the segment $[z'_k,z]$. We define
\begin{align*}
    t_{k+1} := \inf\{t\in [t'_k,1]\ |\ x(t)\in \bigcup_{j\in \{1,\dots,n\} \setminus \{i_1,\dots,i_k\}}\gamma_j\},
\end{align*}
and 
    $z_{k+1} := x(t_{k+1}).$
We also define $i_{k+1}$ to be the integer $i_{k+1}\in \{1,\dots,n\}\setminus \{i_1,\dots,i_k\}$ such that $z_{k+1}\in \gamma_{i_{k+1}}$. Then we define
\begin{align*}
    t'_{k+1} := \sup\{t\in [t_{k+1},\tau]\ |\ x(t)\in \gamma_{i_{k+1}}\},
\end{align*}
and $z'_{k+1}=x(t'_{k+1}).$
All these definitions make sense by adapting the argument of Claim \ref{claim:2} to $t_{k+1}$ and $t'_{k+1}$. We say that the process ends if $i_{k+1}=1$ (and so necessarily $z_{k+1}=z'_{k+1}=z$) otherwise we continue the exploration.

We argue that this process will always end. In fact, during the exploration all $i_k$ are pairwise distinct but there are at most $n$ value for $i_k$. We denote $N\geq 1$ to be the first $k$ such that $i_k=1$ (and so $z_k=z'_k=z$).
We claim that the following properties hold:
\begin{itemize}
 \item $\forall 0\leq k \leq N-1,\  [z'_k, z_{k+1}] \subset \mathcal{C}$.
 \item $\sum_{k=0}^{N-1} d_\chem^\mathcal{C}(z'_k, z_{k+1}) \leq \frac{1}{2}\text{length}(\gamma_1).$
 \item $\forall 1 \leq k \leq N-1, d_\chem^\mathcal{C}(z_k, z'_k) \leq \frac{1}{2}\text{length}(\gamma_{i_k})$.
\end{itemize}
The third property is probably the easiest one. By definition, $z_k$ and $z'_k$ are two points that belong to $\gamma_{i_k}$, so we have a path in $\mathcal{C}$ that joins them just by staying in $\gamma_{i_k}$. Such a path can be made to have a length less than $\frac{1}{2}\text{length}(\gamma_{i_k})$, hence the conclusion.

The first property comes from the fact that initially the point $z'_0=x$ is in the interior of $\gamma_1$ and at the exterior of every other curve $\gamma_j$ for $j\geq 2$ (see Claim \ref{claim:1}). By definition, $z_1$ is the first point along the segment $[z'_0,z']$ that intersects one of the curve $\gamma_i$, hence all the segment $[z'_0,z_1]$ is included in the interior of $\gamma_1$ and in the exterior of every other curve $\gamma_j$ for $j\geq 2$. This proves that the full segment $[z'_0,z_1]$ is included in $\mathcal{C}$ and the same argument apply to any other segment $[z'_k, z_{k+1}]$.

Finally for the second property, we observe thanks to the first property that the sum of all terms $d_\chem^\mathcal{C}(z'_k,z_{k+1})$ is equal to the sum of the Euclidean length of all segments $[z'_k, z_{k+1}]$ which itself is less than the length of the segment $[x,z']$:
\begin{equation}
    \label{eq:grid0}
    \sum_{k=0}^{N-1}d_\chem^\mathcal{C}(z'_k,z_{k+1}) = \sum_{k=0}^{N-1}\text{length}([z'_k,z_{k+1}])\leq \text{length}([x,z]).
\end{equation}
Now since $z$ is the first point of intersection of $S$ with $\gamma_1$, we see that $[x,z]$ is included in the interior of $\gamma_1$.
\begin{equation}
    \label{eq:grid1}
    [x,z]\subset \overline{\text{Int}}(\gamma_1).
\end{equation}
We can deduce from \eqref{eq:grid1} an inequality on the usual Euclidean diameter of the sets involved.
\begin{equation}
    \label{eq:grid2}
    \text{length}([x,z]) = \text{diam}([x,z])\leq \text{diam}(\overline{\text{Int}}(\gamma_1)).
\end{equation}
We recall that for a set $E\subset \R^2$, the diameter of $E$ is the quantity $\diam(E) =\sup_{x,y\in E}\norm{x-y}{}.$
Now, we state the following claim:
\begin{claim}
\label{claim:3}
If $\gamma_1$ is a piecewise $\mathcal{C}^1$ Jordan curve, then $\text{diam}\left(\overline{\text{Int}}(\gamma_1)\right)
\leq \frac{1}{2}\text{length}(\gamma_1).$
\end{claim}
\begin{proof}[Proof of Claim \ref{claim:3}.]
Suppose that $x$ and $y$ realize the diameter of $\overline{\text{Int}}(\gamma_1)$ (such points can be found since $\overline{\text{Int}}(\gamma_1)$ is a compact set). Then we argue that $x$ and $y$ belong to $\gamma_1$. If $x=y$ there is nothing to prove, the diameter is $0$. Otherwise, by contradiction, if say $x\not\in \gamma_1$ we can find $B(x,r)$ a small Euclidean ball of radius $r>0$ centered at $x$, such that $B(x,r)\subset \text{Int}(\gamma_1)$. Now we consider $z=x+\frac{r(x-y)}{2\norm{x-y}{}}$. By definition we have $z\in \text{Int}(\gamma_1)$ however we have $\norm{z-y}{}>\norm{x-y}{}$ contradicting the definition of $x$ and $y$.

We can now assume that $x$ and $y$ belong to $\gamma_1$, we see that $\gamma_1$ is the union of two arcs joining $x$ and $y$, each of this arc is of length at least $\text{diam}(\text{Int}_{\gamma_1})$ (which is the Euclidean distance between $x$ an $y$). This yields the conclusion.
\end{proof}
We return to the proof of Lemma \ref{lemma:grid}. Combining \eqref{eq:grid0}, \eqref{eq:grid2} together with Claim \ref{claim:3}, we get
\begin{equation}
    \label{eq:grid3}
    \sum_{k=0}^{N-1} d_\chem^\mathcal{C}(z'_k, z_{k+1}) \leq \frac{1}{2}\text{length}(\gamma_1).
\end{equation}
which is precisely the second property stated above.

We can finally conclude the proof of Lemma \ref{lemma:grid} by using the triangular inequality. In fact
\begin{align}
\label{eq:19}
    d_\chem^\mathcal{C}(x,z)& \leq \sum_{k=0}^{N-1}d_\chem^\mathcal{C}(z'_k,z_{k+1})+\sum_{k=1}^{N-1}d_\chem^\mathcal{C}(z_{k}, z'_k)\nonumber\\
    &\leq \frac{1}{2}\text{length}(\gamma_1) + \sum_{k=1}^{N-1}\frac{1}{2}\text{length}(\gamma_{i_k})\nonumber\\
    &\leq \frac{1}{2}\sum_{i=1}^n\text{length}(\gamma_i).
\end{align}
As mentioned before \eqref{eq:18} and \eqref{eq:19} together imply \eqref{eq:claim_length} which itself implies the conclusion of the lemma.
\end{proof}

\nocite{*}
\printbibliography

\end{document}